\documentclass[11pt]{amsart}

\usepackage{geometry}                
\usepackage{graphicx}
\usepackage{tikz}
\usepackage{amssymb}
\usepackage{epstopdf}
\usepackage{amsmath}
\usepackage{color}
\usepackage{url}
\usepackage{stmaryrd}
\usepackage{mathtools}
\usepackage{array}

\theoremstyle{definition}

\newtheorem*{ack}{Acknowledgements}
\newtheorem*{notat}{Notation}
\newtheorem*{thm*}{Theorem}
\newtheorem*{rem}{Remark}
\theoremstyle{plain}

\newtheorem{thm}{Theorem}[section]
\newtheorem{lem}[thm]{Lemma}
\newtheorem{prop}[thm]{Proposition}

\numberwithin{equation}{section}
\makeatletter
\newcommand*{\@old@slash}{}\let\@old@slash\slash
\def\slash{\relax\ifmmode\delimiter"502F30E\mathopen{}\else\@old@slash\fi}
\makeatother

\DeclareMathOperator{\Ann}{Ann}

\DeclareMathOperator{\Pic}{Pic}

\DeclareMathOperator{\GL}{GL}
\DeclareMathOperator{\SL}{SL}

\newcommand{\proj}{\mathbb{P}}

\newcommand{\Q}{\mathbb{Q}}
\newcommand{\Z}{\mathbb{Z}}

\newcommand{\C}{\mathbb{C}}

\newcommand{\F}{\mathbb{F}}

\newcommand{\uC}{\mathcal{H}}

\DeclareMathOperator{\End}{End}

\newcommand{\T}{\mathbb{T}}

\newcommand{\loc}{{\mathcal O}}

\DeclareMathOperator{\Spec}{Spec}
\DeclareMathOperator{\Jac}{Jac}

\DeclareMathOperator{\Tors}{Tors}

\title[Serre's Uniformity Conjecture]{Serre's Uniformity Conjecture for Elliptic Curves with Rational Cyclic Isogenies}
\author{Pedro Lemos}
\address{University of Warwick, Mathematics Institute, Coventry, UK}
\email{lemos.pj@gmail.com}

\begin{document}
\maketitle

\begin{abstract}
Let $E$ be an elliptic curve over $\Q$ such that $\End_{\bar{\Q}}(E)=\Z$ and which admits a non-trivial cyclic $\Q$-isogeny. We prove that, for $p>37$, the residual mod $p$ Galois representation $\bar{\rho}_{E,p}:G_{\Q}\rightarrow\GL_2(\F_p)$ is surjective.
\end{abstract}
\section{Introduction}

In \cite{ser1}, Serre famously proved that, given an elliptic curve $E$ defined over a number field $K$ without complex multiplication, there exists a prime number $p_{E,K}$ such that, for any prime $p> p_{E,K}$, the image of the residual mod $p$ Galois representation $\bar{\rho}_{E,p}:G_{\Q}\rightarrow\GL_2(\F_p)$ attached to $E$ is the whole of $\GL_2(\F_p)$. In the very same paper, Serre raised the following question: 
\newline

Given a number field $K$, is there a prime number $p_K$ such that, for any elliptic curve $E$ over $K$ without complex multiplication, the residual mod $p$ Galois representation $\bar{\rho}_{E,p}$ is surjective onto $\GL_2(\F_p)$, whenever $p$ is a prime larger than $p_K$? 
\newline

This conjecture remains open today, but, over the last forty years, there has been a lot of progress towards a proof for $K=\Q$ --- it is believed that, in this case, $p_K=37$. The classification of maximal subgroups of $\GL_2(\F_p)$ plays a central role in the general strategy used to tackle this problem: the aim is to try to show that, for $p$ large enough, there are no elliptic curves without complex multiplication for which the image of $\bar{\rho}_{E,p}$ is contained in any of these maximal subgroups. The maximal subgroups not containing $\SL_2(\F_p)$ are the Borel subgroups, the normalisers of (split and non-split) Cartan subgroups and a few exceptional ones. The exceptional cases were treated by Serre in \cite{ser2}. Mazur, in \cite{maz1}, treated the Borel case, exhibiting all the possible prime degrees of rational isogenies admitted by elliptic curves over $\Q$: for elliptic curves over $\Q$ without complex multiplication, the possible prime degrees of rational isogenies are $2,3,5,7,11,13,17$ and $37$. Finally, Bilu and Parent \cite{bilu} and Bilu, Parent and Rebolledo \cite{rebbil} studied the case of the normaliser of a split Cartan. In \cite{rebbil}, they proved that if $E$ is an elliptic curve over $\Q$ without complex multiplication, and $p\geq 11$ is a prime different from $13$, then the image of $\bar{\rho}_{E,p}$ is not contained in the normaliser of a split Cartan subgroup of $\GL_2(\F_p)$. The verification of the conjecture is thus reduced to the proof that, for $p$ large enough, the image of the mod $p$ Galois representation of any non-CM elliptic curve over $\Q$ is not contained in the normaliser of any non-split Cartan subgroup of $\GL_2(\F_p)$.
\newline

For deep reasons, this last case seems, for the moment, out of reach. However, in the direction of such a result, we prove, in this paper, the following:
\begin{thm}\label{main}
Let $E$ be an elliptic curve over $\Q$ such that $\End_{\bar{\Q}}(E)=\Z$. Suppose, moreover, that $E$ admits a non-trivial cyclic $\Q$-isogeny. Then, for $p>37$, the residual mod $p$ Galois representation $\bar{\rho}_{E,p}:G_{\Q}\rightarrow\GL_2(\F_p)$ is surjective.
\end{thm}

Mazur \cite{maz1} proved that if $p\geq 11$ is a prime different from $13$, then any elliptic curve over $\Q$ with a $\Q$-rational $p$ isogeny has potentially good reduction at every prime $\ell>2$. In order to prove this, Mazur introduced the concept of \emph{formal immersion}, concept which will be central in our discussion. We say that a morphism $f:X\rightarrow Y$ between two schemes is a \emph{formal immersion at a point $P\in X$} if the induced morphism of completed local rings at $P$ \begin{equation*}\hat{f}^*:\hat{\loc}_{Y,f(P)}\rightarrow\hat{\loc}_{X,P}\end{equation*}is surjective. We will briefly describe his method. 
\newline

Say that $E$ is an elliptic curve over $\Q$ and that $C$ is a $\Q$-rational subgroup of $E(\bar{\Q})$ of order $p$, where $p\geq 11$ is a prime different from $13$. Then $(E,C)$ is a representative of an isomorphism class of pairs which, by the moduli interpretation of $X_0(p)$, corresponds to a non-cuspidal rational point $P$ of $X_0(p)$. Moreover, since $E$ is assumed to have potentially multiplicative reduction at a prime $\ell>2$ (for simplicity, we will take $\ell$ to be different from $p$), the reduction of $P$ mod $\ell$ coincides with a cusp. By an application of the Atkin--Lehner involution if necessary, we may assume that this cusp is $\infty$. Mazur, in \cite{maz} and \cite{maz1}, proves two fundamental claims that make the remainder of the argument work: firstly, he proves that $J_0(p):=\Jac(X_0(p))$ has a non-trivial rank $0$ quotient $A$ defined over $\Q$; secondly, he shows that the morphism
\begin{equation*}
f:X_0(p)_{\Z_{\ell}}\rightarrow A_{\Z_{\ell}},
\end{equation*}
defined by composing the Abel--Jacobi map $X_0(p)_{\Z_{\ell}}\rightarrow J_0(p)_{\Z_{\ell}}$ with the natural projection $J_0(p)_{\Z_{\ell}}\rightarrow A_{\Z_{\ell}}$ is a formal immersion at $\infty_{\ell}$, the reduction of $\infty$ to the special fibre. Here, $X_0(p)_{\Z_{\ell}}$ stands for the minimal regular model of $X_0(p)$ over $\Z_{\ell}$, and $A_{\Z_{\ell}}$ and $J_0(p)_{\Z_{\ell}}$ for the N\'eron models of $A$ and $J_0(p)$ over $\Z_{\ell}$. As we have seen, $P$ reduces to $\infty_{\ell}$, which means that $f(P)$ will reduce to $0$ mod $\ell$. However, $f(P)$ is rational, and, since $A(\Q)$ is finite, it must be a torsion point of $A$. Now, $\Tors A(\Q)$ injects, by reduction mod $\ell$, into $A(\F_{\ell})$. Therefore, $f(P)=0$. Let
\begin{equation*}
s_{\infty}:\Spec\Z_{\ell}\rightarrow X_0(p)_{\Z_{\ell}}\quad\text{and}\quad s_P:\Spec\Z_{\ell}\rightarrow X_0(p)_{\Z_{\ell}}
\end{equation*}
be, respectively, the sections corresponding to $\infty$ and $P$. We find that $f\circ s_{\infty}=f\circ s_P$. Therefore,
\begin{equation*}
\hat{s}_{\infty}^*\circ\hat{f}^*=\hat{s}_P^*\circ\hat{f}^*,
\end{equation*}
where $\hat{s}^*_{\infty}, \hat{s}^*_P$ and $\hat{f}^*$ are the induced maps on the completed local rings at $\infty_{\ell}, P_{\ell}$ and $0_{\ell}$, the mod $\ell$ reductions of $\infty, P$ and $0$, respectively. But since $\hat{f}^*$ is surjective, this yields $\hat{s}_{\infty}^*=\hat{s}_P^*$, implying that $P=\infty$. However, we asserted above that $P$ is non-cuspidal, and, therefore, we have a contradiction.
\newline

We might try to apply this strategy to other modular curves, such as $X_{\mathrm{ns}}^+(p)$. However, as it is remarked in \cite{dar}, assuming the Birch and Swinnerton-Dyer conjecture, it can be shown that the Jacobian of $X_{\mathrm{ns}}^+(p)$ does not have a non-trivial rank~$0$ quotient defined over $\Q$. Indeed, by Chen \cite{chen} (see Theorem \ref{cheniso} in section \ref{sec2}),  the Jacobian of $X_{\mathrm{ns}}^+(p)$ is isogenous to the Jacobian of $X_0^+(p^2):=X_0(p^2)/w_{p^2}$, and the $L$-functions of every weight $2$ cusp form of $\Jac(X_0^+(p^2))$ have sign $-1$ in their functional equations. By the Birch and Swinnerton-Dyer conjecture, every non-trivial quotient of $\Jac(X_{\mathrm{ns}}^+(p))$ will then have Mordell--Weil rank $\geq 1$. But, as we have seen, the existence of a non-trivial rank $0$ quotient is necessary for Mazur's method to work. 
\newline

Nevertheless, using a result by Imin Chen \cite{chen}, later generalized by de Smit and Edixhoven \cite{smi}, Darmon and Merel \cite{dar} proved the following result.

\begin{thm}[{\cite[Proposition 7.1]{dar}}]\label{step}
Let $r=2$ or $3$, and let $p>3$ be a prime. There exists a non-trivial optimal quotient $A$ of the Jacobian $J_{0,\mathrm{ns}}^+(r,p)$ of the curve \begin{equation*}X_0(r)\times_{X(1)}X_{\mathrm{ns}}^+(p)\end{equation*} such that $A(\Q)$ is finite. Moreover, the kernel of the canonical projection $J_{0,\mathrm{ns}}^+(r,p)\rightarrow A$ is stable under the action of the Hecke operators $T_n$ for $n$ coprime to $p$.
\end{thm}
With it, they were able to show the following:
\begin{thm}[{\cite[Theorem 8.1]{dar}}]\label{hei}
Let $E$ be an elliptic curve over $\Q$ admitting a $\Q$-rational $r$-isogeny, where $r=2$ or $3$. Suppose that that there exists a prime $p>3$ such that the image of $\bar{\rho}_{E,p}$ is contained in the normaliser of a non-split Cartan subgroup of $\GL_2(\F_p)$. Then $j(E)\in\Z[\frac{1}{p}]$.
\end{thm}
In fact, their methods work not only when $r=2$ or $3$, but whenever $X_0(r)$ has genus $0$; in other words, whenever $r\in\{2,3,5,7,13\}$ (subject to the condition that, in the results aforementioned, $p\notin\{2,3,5,7,13\}$). Therefore, we have the following theorem:
\begin{thm}\label{dm}
Set $\Sigma:=\{2,3,5,7,13\}$. Let $E$ be an elliptic curve over $\Q$ admitting a $\Q$-rational $r$-isogeny, for some $r\in\Sigma$. Suppose that that there exists a prime $p\notin\Sigma$ such that the image of $\bar{\rho}_{E,p}$ is contained in the normaliser of a non-split Cartan subgroup of $\GL_2(\F_p)$. Then $j(E)\in\Z[\frac{1}{p}]$.
\end{thm}

The proof of this theorem is essentially the same as the one presented by Darmon and Merel \cite{dar} for Theorem \ref{hei}. Section \ref{sec2} will provide an outline of it, the details being referred to \cite{dar}.
\newline

There are essentially two steps in the proof of Theorem \ref{main}. The first one is to prove that, under the conditions of Theorem \ref{dm}, we can actually conclude that $j(E)\in\Z$. The second one consists, firstly, in noting that, for $p>37$, the image of the mod $p$ Galois representation of any non-CM elliptic curve over $\Q$ will be either contained in the normaliser of a non-split Cartan subgroup of $\GL_2(\F_p)$, or will be the whole of $\GL_2(\F_p)$; and, secondly, in checking that, for $r\in\{2,3,5,7,13\}$, the non-CM elliptic curves corresponding to the rational points of $X_0(r)$ with integral $j$-invariant have surjective mod $p$ Galois representations. Separately, we check the same thing for the non-cuspidal rational points of $X_0(11), X_0(17)$ and $X_0(37)$. This will yield the theorem.
\begin{notat}
We will denote the normaliser of a split Cartan subgroup of $\GL_{2}(\F_p)$ by $N_{\mathrm{sp}}$, and the normaliser of a non-split Cartan subgroup by $N_{\mathrm{ns}}$. The quotients $X(p)/N_{\mathrm{sp}}$ and $X(p)/N_{\mathrm{ns}}$ will be abbreviated to $X_{\mathrm{sp}}^+(p)$ and $X_{\mathrm{ns}}^+(p)$, respectively. Moreover, we will denote the fibre product $X_0(r)\times_{X(1)}X_{\mathrm{sp}}^+(p)$ by $X_{0,\mathrm{sp}}^+(r,p)$, and, similarly, the fibre product $X_0(r)\times_{X(1)}X_{\mathrm{ns}}^+(p)$ by $X_{0,\mathrm{ns}}^+(r,p)$.
\end{notat}
\begin{ack}
I would like to thank Samir Siksek for all his helpful advice and suggestions. Also, I want to thank the referee for the useful comments and corrections. Finally, I would like to express my gratitude to the EPSRC for the financial support. 
\end{ack}

\section{Proof of Theorem \ref{main}}\label{msec}




One of the fundamental observations of this paper is the following improvement of Theorem \ref{dm}:

\begin{prop}\label{fund}
Set $\Sigma:=\{2,3,5,7,13\}$. Let $E$ be an elliptic curve over $\Q$ admitting a $\Q$-rational $r$-isogeny for some $r\in\Sigma$. Suppose that there exists a prime $p\notin\Sigma$ such that the image of $\bar{\rho}_{E,p}$ is contained in the normaliser of a non-split Cartan subgroup of $\GL_2(\F_p)$. Then $j(E)\in\Z$.
\end{prop}

This proposition is actually a corollary of the following result.
\begin{prop}\label{good}
Let $E$ be an elliptic curve defined over $\Q$ and $p\geq 5$ a prime number such that $\bar{\rho}_{E,p}(G_{\Q})$ is contained in the normaliser of a non-split Cartan subgroup of $\GL_2(\F_p)$. Then,  for any prime $\ell\not\equiv\pm 1\pmod{p}$, the elliptic curve $E$ has potentially good reduction at $\ell$.
\end{prop}
\begin{proof}
Suppose $\bar{\rho}_{E,p}(G_{\Q})$ is contained in the normaliser $N_{\mathrm{ns}}$ of a non-split Cartan subgroup $C_{\mathrm{ns}}$ of $\GL_2(\F_p)$. Let $\ell$ be a prime of potentially multiplicative reduction. For the remainder of this proof, we will write $E$ for $E_{\Q_{\ell}}$, the elliptic curve obtained from $E$ by extension of scalars to $\Q_{\ell}$. Also, we fix an embedding $\bar{\Q}\hookrightarrow\bar{\Q}_{\ell}$, which amounts to a choice of decomposition subgroup $G_{\Q_{\ell}}\hookrightarrow G_{\Q}$ over $\ell$.

Since $E$ has potentially multiplicative reduction, it is a quadratic twist of a Tate curve $E_q$, $q\in\Q_{\ell}^{\times}$. Let $\psi$ be the quadratic character associated to this twist ($\psi$ may well be the trivial character). Then $\bar{\rho}_{E,p}\cong\bar{\rho}_{E_q,p}\otimes\psi$. Since we have
\begin{equation*}
\bar{\rho}_{E_q,p}\sim \begin{pmatrix}\chi_{p} & *\\ 0 & 1\end{pmatrix},
\end{equation*}
where $\chi_{p}:G_{\Q_{\ell}}\rightarrow \F_p^{\times}$ is the mod $p$ cyclotomic character, we conclude that
\begin{equation*}
\bar{\rho}_{E,p}\sim\begin{pmatrix}\psi\chi_{p} & *\\ 0 & \psi\end{pmatrix}.
\end{equation*}

Now, note that $C_{\mathrm{ns}}$, as a subgroup of $\GL_2(\F_{p^2})$, is conjugate to the subgroup
\begin{equation}\label{norm}
\left\{\begin{pmatrix} a & 0\\ 0 & a^p\end{pmatrix}:a\in\F_{p^2}^{\times}\right\}\subseteq\GL_2(\F_{p^2})
\end{equation}

Since $[N_{\mathrm{ns}}:C_{\mathrm{ns}}]=2$, we have $\bar{\rho}_{E,p}(\sigma)^2\in C_{\mathrm{ns}}$, for all $\sigma\in G_{\Q_{\ell}}$. Also, since $\psi$ is quadratic, the eigenvalues of $\bar{\rho}_{E,p}(\sigma)^2$ are $\chi_p(\sigma)^2$ and $1$. It then follows from (\ref{norm}) that $\chi_p(\sigma)^2=1$, for all $\sigma\in G_{\Q_{\ell}}$. If $\ell=p$, then we know that $\chi_p$ surjects onto $\F_p^{\times}$, which forces $p\leq 3$. If $\ell\neq p$, then we have $\ell^2\equiv 1\pmod{p}$.
\end{proof}

\begin{proof}[Proof of Proposition \ref{fund}]
Let $E$ be an elliptic curve and $p$ a prime as in the statement of the proposition. Then we already know, due to Theorem \ref{dm}, that $j(E)\in\Z[\frac{1}{p}]$. 
Note that $E$ and $p$  satisfy the conditions of Proposition \ref{good}. It follows that $E$ has potentially good reduction at $p$. Therefore, $j(E)\in\Z$.
\end{proof}

With Proposition \ref{fund} proven, we have the most important ingredients for the proof of the main result of this paper.



\begin{proof}[Proof of Theorem \ref{main}]
Let $E$ is an elliptic curve over $\Q$ without CM which admits a cyclic $\Q$-rational isogeny, which we may assume to be of prime degree $r$. By Mazur \cite{maz1}, we have $r\in\{2,3,5,7,11,13,17,37\}$. Suppose, for the sake of contradiction, that there exists a prime number $p>37$ such that $\bar{\rho}_{E,p}(G_{\Q})\neq \GL_2(\F_p)$. For each prime number $r$, define
\begin{equation*}
S_r=\{j(P): P\in Y_0(r)(\Q)\}.
\end{equation*}


We will distinguish two cases: we can either have $r\in\{2,3,5,7,13\}$, or $r\in\{11,17,37\}$. For $r=11, 17$, the modular curve $X_0(r)$ is an elliptic curve over $\Q$ of rank $0$; and for $r=37$, $X_0(r)$ is a curve of genus $2$ whose Jacobian has rank $0$. This makes it easy to determine the rational points of $X_0(r)$ for $r\in\{11,17,37\}$, which are, in fact, known. From {\cite[p. 98]{crem}}, we obtain

\begin{equation*}
\begin{aligned}
S_{11} &= \{-11\cdot131^3, -2^{15},  -11^2\};\\
S_{17} &= \left\{-\frac{17^2\cdot101^3}{2}, -\frac{17\cdot373^3}{2^{17}}\right\};\\
S_{37} &= \{-7\cdot137^3\cdot2083^3, -7\cdot 11^3 \}.
\end{aligned}
\end{equation*}
Therefore, if $r\in\{11,17,37\}$, the $j$-invariant of $E$ is one of the values in \mbox{$S_{11}\cup S_{17}\cup S_{37}$}. Since any two elliptic curves over $\Q$ without CM and with the same $j$-invariant are related by a quadratic twist, the surjectivity of $\bar{\rho}_{E,p}$ only depends on the $j$-invariant. The LMFDB~\cite{lmfdb} provides a long list of elliptic curves over $\Q$, together with information about the surjectivity of the mod $p$ Galois representations attached to them, such as the largest non-surjective prime --- which is computed using an algorithm of Sutherland \cite{sut}. For each of the seven values in $S_{11}\cup S_{17}\cup S_{37}$, we found an elliptic curve over $\Q$ in this database with this $j$-invariant, we verified that $-2^{15}$ is the only CM $j$-invariant, and checked that the largest non-surjective prime of each of the other six $j$-invariants is $\leq 37$. Therefore, if an elliptic curve $E$ defined over $\Q$ without CM admits a $\Q$-rational isogeny of degree $r\in\{11,17,37\}$, then the image of $\bar{\rho}_{E,p}$ is $\GL_2(\F_p)$ for all $p>37$.
\newline

Keeping up with the notation introduced at the beginning of this proof, we suppose that $r\in\{2,3,5,7,13\}$. Now, $X_0(r)$ is a smooth curve of genus $0$ with rational points (the cusps, for instance), which means that $X_0(r)$ will have infinitely many rational points. Therefore, we will not be able to use the same strategy we applied to treat the case $r\in\{11,17,37\}$. Instead, we are going to start by showing that if $\bar{\rho}_{E,p}(G_{\Q})\neq\GL_2(\F_p)$, then $j(E)\in\Z$. We are only a step away from proving this. After the successive works mentioned in the introduction, we have the following theorem: 
\begin{thm}[Bilu--Parent--Rebolledo {\cite{rebbil}}, Mazur \cite{maz1}, Serre {\cite[Lemme~18]{ser2}}]
Let $E$ be an elliptic curve over $\Q$ such that $\End_{\bar{\Q}}(E)=\Z$. If $p>37$ is a prime such that $\bar{\rho}_{E,p}(G_{\Q})\neq\GL_2(\F_p)$, then the image of $\bar{\rho}_{E,p}$ is contained in the normaliser of a non-split Cartan subgroup of $\GL_2(\F_p)$.
\end{thm}

This result, together with Proposition \ref{fund}, yields $j(E)\in\Z$. 
\newline

This would be of no use for us if $S_r\cap\Z$ were infinite; but it turns out that, for $r\in\{2,3,5,7,13\}$, this set is finite. In order to prove this, we can directly compute these sets, and that is what we will do now.
\newline

For an appropriate choice of local coordinate $t$ on $X_0(r)\cong\proj^1_{\Q}$, the $j$-invariant map is explicitly given by a map of the form
\begin{equation*}
t\mapsto \frac{f(t)}{t},
\end{equation*}
where $f(t)\in\Z[t]$ is a monic polynomial of degree $r+1$ and non-zero constant term. The values of $f(t)$ can be found in the following table:
\begin{center}
\begin{tabular}{|c|c|}
\hline
$r$ & $f(t)$\\
\hline \hline
$2$ & $(t+16)^3$\\
\hline
$3$ & $(t+27)(t+3)^3$\\
\hline
$5$ & $(t^2+10t+5)^3$\\
\hline
$7$ & $(t^2+5t+1)^3(t^2+13t+49)$\\
\hline
$13$ & $(t^4+7t^3+20t^2+19t+1)^3(t^2+5t+13)$\\
\hline
\end{tabular}
\end{center}
For $r=2,3$, they are easy to check (see {\cite[pp. 179-180]{antw}}); for $r=5,7,13$, we refer to {\cite[p. 54]{dahm}}. If $t$ corresponds to a $\Q$-point in $X_0(r)$, then we can write $t=a/b$, where $a,b\in\Z$, $b>0$ and $\gcd(a,b)=1$. Now, if the $j$-invariant of this point is integral, we have
\begin{equation*}
\frac{b^{r+1}f(a/b)}{ab^r}\in\Z.
\end{equation*}
Therefore, $b\mid b^{r+1}f(a/b)$. But $b^{r+1}f(a/b)=a^{r+1}+bG(a,b)$, where $G\in\Z[s,t]$ is a homogeneous polynomial. Since $\gcd(a,b)=1$, we must have $b=1$. Hence, $t\in\Z$ and $t$ must divide the constant term of $f(t)$. Substituting $t$ in $f(t)/t$ by all integers that divide the constant term of $f(t)$, we obtain $S_r\cap\Z$: 
\begin{equation*}
\begin{aligned}
S_2\cap\Z = \{ &-3^3\cdot 5^3,-2^2\cdot7^3,-2^4\cdot3^3,-2^6,0,2^7,2^6\cdot3^3,2^4\cdot5^3,2^{11},2^2\cdot3^6,2^7\cdot3^3, 17^3,\\ &2^6\cdot5^3,2^5\cdot7^3,2^5\cdot3^6,2^4\cdot3^3\cdot5^3,2^4\cdot17^3,2^3\cdot31^3,2^3\cdot3^3\cdot11^3,2^2\cdot3^6\cdot7^3,\\ &2^2\cdot5^3\cdot13^3,2\cdot127^3,2\cdot3^3\cdot43^3,3^3\cdot5^3\cdot17^3,257^3\};\\
S_3\cap\Z = \{ &-2^4\cdot11^6\cdot13,-2^{15}\cdot3\cdot5^3,-2^4\cdot3^2\cdot13^3,-2^4\cdot13,0,2^4\cdot3^3,2^8\cdot7,2^4\cdot3^3\cdot5,\\ &2^8\cdot3^3,2^4\cdot3^3\cdot5^3,2^8\cdot3^2\cdot7^3,2^4\cdot3\cdot5\cdot41^3,2^8\cdot7\cdot61^3\};\\
S_5\cap\Z =\{ &-2^6\cdot719^3, -2^6\cdot5\cdot19^3, 2^6, 2^6\cdot5^2, 2^{12}, 2^{12}\cdot5^2, 2^{12}\cdot5\cdot11^3,2^{12}\cdot211^3\};\\
S_7\cap\Z=\{ &-3^3\cdot37\cdot719^3, -3^3\cdot5^3, 3^3\cdot37, 3^2\cdot7^4, 3^3\cdot5^3\cdot17^3, 3^2\cdot7\cdot2647^3\};\\
S_{13}\cap\Z=\{ &-2^6\cdot3^2\cdot4079^3, 2^6\cdot3^2, 2^{12}\cdot3^3\cdot19, 2^{12}\cdot3^3\cdot19\cdot991^3\}.
\end{aligned}
\end{equation*}
Resorting once again to the elliptic curve database of the LMFDB \cite{lmfdb}, we verified that the largest non-surjective prime of each of the non-CM $j$-invariants in \begin{equation*}(S_2\cup S_3\cup S_5\cup S_7\cup S_{13})\cap \Z\end{equation*} is not larger than $37$. This concludes the proof of Theorem \ref{main}.\end{proof}
%


\begin{rem}
During the proof of Theorem \ref{main}, we computed the sets $S_r\cap\Z$, for $r\in\{2,3,5,7,13\}$, using an explicit description of the $j$-invariant map. There is, however, a nice proof of the finiteness of the sets $S_r\cap\Z$, pointed out to me by Samir Siksek. Since this proof is interesting in its own right, we include it here.
\begin{prop}\label{finite}
Let $p$ be a prime number. Then the set $S_p\cap\Z$ is finite.
\end{prop}
\begin{proof}
If the genus of $X_0(p)$ is at least $1$, then it is known that there are only finitely many points in $X_0(p)(\Q)$. Therefore, we may assume that the genus of $X_0(p)$ is $0$, i.e., that $p\in\{2,3,5,7,13\}$.

Fix an isomorphism $\psi:\proj^1_{\Z[\frac{1}{p}]}\rightarrow X_0(p)_{\Z[\frac{1}{p}]}$ over $\Z[\frac{1}{p}]$, and choose projective coordinates in such a way that $(0:1)$ is mapped to the cusp $0$ and $(1:0)$ to the cusp $\infty$. The $j$-invariant map is then a morphism 
\begin{equation*}
j:\proj^1_{\Z[\frac{1}{p}]}\rightarrow\proj^1_{\Z[\frac{1}{p}]}
\end{equation*}
mapping $(1:0)$ and $(0:1)$ to $(1:0)$. For a point in $X_0(p)$ to have integral $j$-invariant, it is necessary that its image under the $j$-invariant map does not intersect $(1:0)$ in any special fibre of $\proj^1_{\Z[\frac{1}{p}]}$. Therefore, such a point must not intersect $(0:1)$ nor $(1:0)$ in any special fibre of $\proj^1_{\Z[\frac{1}{p}]}\cong X_0(p)_{\Z[\frac{1}{p}]}$. This means that it must be of the form $(p^k:1)$, for some $k\in\Z$. 

Consider now $X_0(p)(\C_{p})$ and $\proj^1(\C_{p})$ equipped with the $p$-adic topology. Consider the $\Q$-isomorphism between $X_0(p)$ and $\proj^1_{\Q}$ obtained by restricting the isomorphism of the paragraph above to the general fibres. Let $B$ be the open ball of radius $1/p$ centered at the point $(1:0)$ of $\proj^1(\C_{p})$; the integral points of $\proj^1(\C_p)$ lie outside of $B$. Since our isomorphism between $X_0(p)$ and $\proj^1_{\Q}$ and the $j$-invariant morphism $j:X_0(p)(\C_{p})\rightarrow \proj^1(\C_p)$ are $p$-adically continuous, $U:=j^{-1}(B)$ is an open subset of $\proj^1(\C_p)$ containing $(1:0)$ and $(0:1)$. Clearly, among the points $(p^k:1)$, $k\in\Z$, only finitely many lie outside of $U$. This concludes the proof of the lemma.
\end{proof}
\end{rem}

\section{Sketch Proof of Theorem \ref{dm}}\label{sec2}

As the title indicates, the purpose of this section is to provide a sketch proof of Theorem~\ref{dm}. The proof is, \emph{mutatis mutandis}, the same as the proof of {\cite[Theorem 8.1]{dar}}, and the details of what follows can be found in \cite{dar}.
\newline

Let $r$ be an integer in the set $\{1,2,3,5,7,13\}$, and $p$ a prime not in there. Write $d:X_0(rp^2)\rightarrow X_0(rp)$ for the degeneracy map that, at the level of $\bar{\Q}$-points, is given by
\begin{equation*}
(E,C)\mapsto (E,C[rp]).
\end{equation*}
By pull-back, we obtain a homomorphism $d^*:\Pic(X_0(rp))\rightarrow\Pic(X_0(rp^2))$. Let $\pi:X_0(rp^2)\rightarrow X_0(r)\times_{X(1)} X_0^+(p^2)$ be the projection morphism. For ease of notation, we will denote the curve $X_0(r)\times_{X(1)}X_0^+(p^2)$ by $X_0^+(rp^2)$. Consider the homomorphism
\begin{equation*}
\pi_*\circ d^*:\Pic(X_0(rp))\rightarrow \Pic(X_0^+(rp^2)).
\end{equation*}
We define the \emph{$p$-old part of} $\Pic(X_0^+(rp^2))$ to be 
\begin{equation*}
\Pic(X_0^+(rp^2))^{p\text{-}\mathrm{old}}:=\pi_*\circ d^*(\Pic(X_0(rp))),
\end{equation*}
and an element of $\Pic(X_0^+(rp^2))^{p\text{-}\mathrm{old}}$ is called a \emph{$p$-old divisor}.
\newline

The \emph{$p$-new part} is defined as the following quotient:
\begin{equation*}
\Pic(X_0^+(rp^2))^{p\text{-}\mathrm{new}}:=\Pic(X_0^+(rp^2))/\Pic(X_0^+(rp^2))^{p\text{-}\mathrm{old}}.
\end{equation*}
When $r=1$, we shorten $p$-old and $p$-new to old and new, respectively.
\newline

As the homomorphisms $d^*$ and $\pi_*$ preserve degrees of divisors, this whole discussion can be reproduced in terms of the Jacobian varieties of $X_0(p), X_0(p^2)$ and $X_0^+(p^2)$. We then obtain the old and new parts of $\Jac(X_0^+(p^2))$, which are abelian varieties in their own right.
\newline

In \cite{chen}, Chen proved the following result:

\begin{thm}[{\cite[Theorem 1]{chen}}]\label{cheniso}
The Jacobian of $X_{\mathrm{ns}}^+(p)$ is isogenous to the new part of the Jacobian of $X_0^+(p^2)$.
\end{thm}

However, his proof, making use of the Selberg trace formula, was not constructive. It was later proven by Chen \cite{chen2} that a construction by Darmon and Merel \cite{dar}, that we now reproduce, is, in fact, an explicit construction of Chen's isogeny.
\newline

In order to describe this construction, we will introduce and recall some notation. Recall that we set $N_{\mathrm{sp}}$ to be the normaliser of a split Cartan subgroup of $\GL_2(\F_p)$, and $N_{\mathrm{ns}}$ the normaliser of a non-split Cartan subgroup. Also, denote by $B^+$ and $B^-$ the subgroups of $\GL_2(\F_p)$ consisting of upper triangular matrices and lower triangular matrices, respectively. Define
\begin{equation*}
X'(p):=X(p)/(N_{\mathrm{sp}}\cap N_{\mathrm{ns}}).
\end{equation*}
For $r\in\{1,2,3,5,7,13\}$, there are two natural projections associated to $X_0(r)\times_{X(1)}X'(p)$:
\begin{equation*}
X_0(r)\times_{X(1)}X'(p)\rightarrow X_{0,\mathrm{sp}}^+(r,p)\quad\text{and}\quad X_0(r)\times_{X(1)}X'(p)\rightarrow X_{0,\mathrm{ns}}^+(r,p).
\end{equation*}
We obtain a correspondence $X_{0,\mathrm{sp}}^+(r,p)\vdash X_{0,\mathrm{ns}}^+(r,p)$, which gives rise to a homomorphism 
\begin{equation*}
\phi:\Jac(X_{\mathrm{0,sp}}^+(r,p))\rightarrow \Jac(X_{0,\mathrm{ns}}^+(r,p)).
\end{equation*}
Since we have an isomorphism between $X_{\mathrm{sp}}^+(p)$ and $X_0^+(p^2)$, we can substitute $\Jac(X_{0,\mathrm{sp}}^+(r,p))$ by the Jacobian of $X_0^+(rp^2)$ in the homomorphism $\phi$. Moreover, 
\begin{lem}[{\cite[Lemma 6.2 (a)]{dar}}]\label{triv}
Under $\phi$, the image of the $p$-old part of $\Jac(X_0^+(rp^2))$ is trivial.
\end{lem}
\begin{proof}
Define the $p$-old part of $\Pic(X_{0,\mathrm{ns}}^+(r,p))$ as the image of \begin{equation*}p^*:\Pic(X_0(r))\rightarrow\Pic(X_{0,\mathrm{ns}}^+(r,p)),\end{equation*} the pull-back of the projection to $X_0(r)$. Since $X_0(r)$ has genus $0$, the $p$-old part of $\Jac(X_{0,\mathrm{ns}}^+(r,p))$ is trivial. Therefore, in order to prove the lemma, we only need to show that the image under $\phi$ of an old divisor is an old divisor.
\newline

Under our isomorphism between $X_{0,\mathrm{sp}}^+(r,p)$ and $X_0^+(rp^2)$, an old divisor in $X_{0,\mathrm{sp}}^+(r,p)$ has an inverse image in $X_0(r)\times_{X(1)}X(p)$ which is the image of a $B^+$ and a $B^-$-invariant divisor. Since \begin{equation*}N_{\mathrm{ns}}B^+=N_{\mathrm{ns}}B^-=\GL_2(\F_p),\end{equation*} the image in $X_{0,\mathrm{ns}}^+(r,p)$ of such a divisor is $\GL_2(\F_p)$-invariant, meaning that it must be an old divisor of $X_{0,\mathrm{ns}}^+(r,p)$.
\end{proof}
We then obtain a homomorphism
\begin{equation*}
\Jac(X_0^+(rp^2))^{p\text{-}\mathrm{new}}\rightarrow \Jac(X_{0,\mathrm{ns}}^+(r,p)),
\end{equation*}
which, by Chen \cite{chen2}, is precisely Chen's isogeny. Moreover, it follows from this explicit description that if $n\geq 1$ is an integer coprime to $p$, then Chen's isogeny commutes with the Hecke operators $T_n$. Summing up,

\begin{thm}[{\cite[Theorem 6.1]{dar}}]
If $r\in\{1,2,3,5,7,13\}$, there is an isogeny between $J_{0,\mathrm{ns}}^+(r,p)$ and $\Jac(X_0^+(rp^2))^{p\text{-}\mathrm{new}}$ which, for any integer $n\geq 1$ coprime to $p$, commutes with the action of $T_n$.
\end{thm}

The existence of such an isogeny is of fundamental importance to prove
\begin{thm}[{\cite[Proposition 7.1]{dar}}]\label{win}
Let $r\in\{2,3,5,7,13\}$ and let $p$ be a prime number not in there. There exists a non-trivial optimal quotient $A$ of $J_{0,\mathrm{ns}}^+(r,p)$, defined over $\Q$, such that $A(\Q)$ is finite. Moreover, if $n\geq 1$ is an integer coprime to $p$, the kernel of the canonical projection $J_{0,\mathrm{ns}}^+(r,p)\rightarrow A$ is stable under the Hecke operators $T_n$.
\end{thm}
\begin{proof}[Sketch proof]
As we have seen, there is an isogeny between $J_{0,\mathrm{ns}}^+(r,p)$ and $\Jac(X_0^+(rp^2))^{p\text{-new}}$ which commutes with the Hecke operators $T_n$ whenever $n$ is coprime to $p$. Thus, we are reduced to prove the result for $\Jac(X_0^+(rp^2))^{p\text{-new}}$. Note that $\Jac(X_0^+(rp^2))^{\mathrm{new}}$ is an optimal quotient of $\Jac(X_0^+(rp^2))^{p\text{-new}}$. Given an optimal quotient $B$ of $J_0^{\mathrm{new}}(N)$, and writing $p:J_0^{\mathrm{new}}(N)\rightarrow B$ for the canonical projection, we know, by a result of Ribet (see, for example, section $2$ of \cite{maz1}), that $\ker p$ is stable under the action of Hecke operators. Therefore, the same will be true for any optimal quotient of $\Jac(X_0^+(rp^2))^{\mathrm{new}}$. Hence, we only need to show that $\Jac(X_0^+(rp^2))^{\mathrm{new}}$ has a non-trivial optimal quotient with only finitely many rational points.  
\newline

Consider the geodesic in $\uC^*$ that connects $0$ to $\infty$ and the path $e$ it defines in $X_0^+(rp^2)$. According to the Drinfeld--Manin theorem, $e\in H_1(X_0^+(rp^2),\Q)$. Moreover, it is fixed under complex conjugation. Hence, $e\in H_1(X_0^+(rp^2),\Q)^+$. Write $e'$ for the image of $e$ in $H_1(\Jac(X_0^+(rp^2))^{\mathrm{new}}(\C),\Q)^+$, and $I_e$ for $\Ann_{\T_{\Z}}(e')$. Set now
\begin{equation*}
A:=\Jac(X_0^+(rp^2))^{\mathrm{new}}/I_e\Jac(X_0^+(rp^2))^{\mathrm{new}}.
\end{equation*}
It remains to prove that $A$ is non-trivial and that $A(\Q)$ is a finite set. For the proof of these two claims, we refer to the paper of Darmon and Merel \cite{dar}.
\end{proof}
We are now in position to apply a variant of the argument by Mazur that was described in the introduction. For details, we refer, once again, to \cite{dar}. In a very informal way, the proof goes as follows. It can be proven that, for any prime $\ell\equiv\pm 1\pmod{p}$, the composition of the maps $X_{0,\mathrm{ns}}^+(r,p)_{\Z_{\ell}}\rightarrow J_{0,\mathrm{ns}}^+(r,p)_{\Z_{\ell}}$ and $J_{0,\mathrm{ns}}^+(r,p)_{\Z_{\ell}}\rightarrow A_{\Z_{\ell}}$ is a formal immersion at $\infty_{\ell}$, where $\infty\in X_{0,\mathrm{ns}}^+(r,p)$ is one of the cusps.  If $E$ is an elliptic curve as in the statement of the theorem, then it gives rise to a $\Q$-rational point $P$ in $X_{0,\mathrm{ns}}^+(r,p)$. If it were the case that $j(E)\notin\Z[\frac{1}{p}]$, then there would exist a prime $\ell\neq p$ dividing the denominator of $j(E)$. Therefore, the $\Z_{\ell}$-section corresponding to the point $P$ would intersect a cusp (which we may assume to be $\infty$) on the special fibre over $\ell$. Since all the cusps of $X_{0,\mathrm{ns}}^+(r,p)$ are only defined over $\Q(\zeta_p)^+$, we conclude that $\ell\equiv\pm 1\pmod{p}$. It can be proven that the image of $P$ in $A$ is torsion (see {\cite[Lemma~8.3.]{dar}}). In a manner similar to the one presented in the introduction, it is now easy to derive a contradiction from the conjunction of this and the fact that the composition of the maps $X_{0,\mathrm{ns}}^+(r,p)_{\Z_{\ell}}\rightarrow J_{0,\mathrm{ns}}^+(r,p)_{\Z_{\ell}}$ and $J_{0,\mathrm{ns}}^+(r,p)_{\Z_{\ell}}\rightarrow A_{\Z_{\ell}}$ is a formal immersion at $\infty_{\ell}$.


\bibliography{Serre_Cyclic}{}
\bibliographystyle{amsplain}

\end{document}